\documentclass[10pt, a4paper]{article}
\usepackage{amssymb, amsfonts, amsthm, amsmath, amssymb}
\usepackage{float}
\usepackage{graphics}
\usepackage[T1]{fontenc}
\usepackage{cite}
\usepackage{tikz}
\usepackage{longtable}
\usepackage{breqn}

\usepackage[colorlinks,
linkcolor=blue, 
anchorcolor=blue, 
citecolor=blue
]{hyperref}

\usepackage[symbol]{footmisc}

\usepackage{enumerate}
\usepackage{indentfirst}
\usepackage{geometry}
\allowdisplaybreaks

\DeclareMathOperator{\mmp}{mmp}
\DeclareMathOperator{\rlmax}{rlmax}
\DeclareMathOperator{\rlmin}{rlmin}
\DeclareMathOperator{\lrmax}{lrmax}
\DeclareMathOperator{\lrmin}{lrmin}
\newtheorem{thm}{Theorem}[section]

\newtheorem{prop}[thm]{Proposition}

\newtheorem{rem}[thm]{Remark}

\usepackage{caption}

\begin{document}
\captionsetup[figure]{labelfont={bf}, labelformat={default}, labelsep={period}, name={Fig.}}
\begin{center}
	{\large \bf Distribution of maxima and minima statistics on alternating permutations, Springer numbers, and avoidance of flat POPs}
\end{center}

\begin{center}
	Tian Han$^{1}$,
	Sergey Kitaev$^{2}$,  and
	Philip B. Zhang$^{3}$\footnote[1]{
		Corresponding author.}
	\\[6pt]

	$^{13}$College of Mathematical Science \& Institute of Mathematics and Interdisciplinary Sciences\\
	Tianjin Normal University, Tianjin  300387, P. R. China\\[6pt]

	$^{2}$Department of Computer and Information Sciences \\
	University of Strathclyde, 26 Richmond Street, Glasgow G1 1XH, UK\\[6pt]

	Email:  $^{1}${\tt hantian.hhpt@qq.com},
	$^{2}${\tt sergey.kitaev@strath.ac.uk},
	$^{3}${\tt zhang@tjnu.edu.cn}
\end{center}

\centerline{\textbf{Abstract.}}
In this paper, we find distributions of the left-to-right maxima, right-to-left maxima, left-to-right minima and right-to-left-minima statistics on up-down and down-up permutations of even and odd lengths.
We recover and generalize a result by Carlitz and Scoville, obtained in 1975, stating that the distribution of left-to-right maxima on down-up permutations of even length is given by $(\sec (t))^{q}$. We also derive the joint distribution of the maxima (resp., minima) statistics, extending the scope of the respective results of Carlitz and Scoville, who obtain them in terms of certain systems of PDEs and recurrence relations. To accomplish this, we generalize a result of Kitaev and Remmel by deriving joint distributions involving non-maxima (resp., non-minima) statistics. Consequently, we refine classic enumeration results of Andr\'e by introducing new $q$-analogues and $(p,q)$-analogues for the number of alternating permutations.

Additionally, we verify Callan's conjecture (2012) that up-down permutations of even length fixed by reverse and complement are counted by
the Springer numbers, thereby offering another combinatorial interpretation of these numbers. Furthermore, we propose two $q$-analogues and a $(p,q)$-analogue of the Springer numbers. Lastly, we enumerate alternating permutations that avoid certain flat partially ordered patterns.


\noindent
\textbf{AMS Classification 2020:} 05A05, 05A15

\section{Introduction}\label{sec1}

\noindent
A permutation of length $n$, or an $n$-permutation, is a rearrangement of the set $[n]:=\{1, 2, \ldots, n\}$.  Denote by  $S_n$  the set of permutations of $[n]$. For $\pi \in S_n$, let $\pi^r=\pi_n\pi_{n-1}\cdots\pi_1$ and $\pi^c= (n + 1 - \pi_1)(n + 1 - \pi_2)\cdots (n + 1-\pi_n)$
denote the {\em reverse} and {\em complement} of $\pi$, respectively. Then $\pi^{rc}=(n + 1-\pi_n)(n + 1-\pi_{n-1})\cdots (n + 1-\pi_1)$. A permutation $\pi_1\pi_2\cdots\pi_n\in S_n$ avoids a {\em pattern} $p_1p_2\cdots p_k\in S_k$  if there is no subsequence $\pi_{i_1}\pi_{i_2}\cdots\pi_{i_k}$ such that $\pi_{i_j}<\pi_{i_m}$ if and only if $p_j<p_m$. For example, the permutation $32154$ avoids the pattern $231$. The area of permutation patterns has attracted much attention in the literature (see \cite{Kitaev2011} and references therein).

We say that $\pi=\pi_1\cdots\pi_n\in S_n$
is an {\em up-down} (resp., {\em down-up}) permutation if it is of the form $\pi_1<\pi_2>\pi_3<\pi_4>\pi_5<\cdots$ (resp., $\pi_1>\pi_2<\pi_3>\pi_4<\pi_5>\cdots$). By ``alternating permutations'', one typically refers to down-up permutations. However, following \cite{KitRem2012}, we define alternating permutations to include both up-down and down-up permutations. Let $UD_n$ (resp., $DU_n$) denote the set of all up-down (resp., down-up) permutations in $S_n$.

Of interest to us are the following classical permutation statistics. For $1\leq i\leq n$, $\pi_i$ is a \emph{right-to-left maximum} (resp., \emph{right-to-left minimum}) in $\pi$ if  $\pi_i$ is greater (resp., smaller) than any element to its right. Note that $\pi_n$ is always a right-to-left maximum and a right-to-left minimum. Denote by $\rlmax(\pi)$ and $\rlmin(\pi)$ the number of right-to-left maxima and right-to-left minima in $\pi$, respectively. We  define \emph{left-to-right maxima} (resp., \emph{left-to-right minima}) in a permutation $\pi$, the number of which is denoted by $\lrmax(\pi)$ (resp., $\lrmin(\pi)$), in a similar way. For example, if $\pi=34152$ then $\lrmax(\pi)=3$ and $\lrmin(\pi)=\rlmin(\pi)=\rlmax(\pi)=2$.

\subsection{Euler numbers}\label{Euler-numbers-subsec}
\noindent The {\em Euler numbers} $E_n$ are defined by the exponential generating function
\begin{equation}\label{E_n}
	E(t):=\sum_{n=0}^{\infty}E_n\frac{t^n}{n!}=\sec t+\tan t.
\end{equation}
The numbers $E_{2n}$ are also called {\em secant numbers}, and  the numbers $E_{2n+1}$ are called {\em tangent numbers}. The Euler numbers satisfy the recurrence \\[-3mm]
$$E_{n+1}=\frac{1}{2}\sum_{k=0}^{n}{n\choose k}E_{n-k}E_{k},$$
for $n\geq 1$, with initial condition $E_0=E_1=1$. Andr\'e \cite{Andre1,Andre2} showed that $E_n$ enumerates  down-up permutations (or, by applying the complement, up-down permutations). The sequence of Euler numbers starts as $1,1,1,2,5,16,61,272,1385,\ldots$; see sequence A000111 in \cite{oeis}. Other combinatorial objects are enumerated by the Euler numbers (see \cite{Sokal2020} and references therein).

\subsection{Springer numbers}
\noindent The {\em Springer numbers} $\mathcal{S}_n$ are defined by the exponential generating function \cite{Gla1898,Gla1914,Springer1971} \\[-3mm]
\begin{equation}\label{Springer numbers}
	\frac{1}{\cos t-\sin t}=\sum_{n=0}^{\infty}\mathcal{S}_n\frac{t^n}{n!}
\end{equation}
and the sequence of Springer numbers starts as $1, 1, 3, 11, 57, 361, 2763, \ldots$; see sequence A001586 in \cite{oeis}.
Arnol'd \cite{Arnold} showed that $\mathcal{S}_n$ enumerates a signed-permutation analogue of the alternating permutations involving the notion of a ``snakes of type $B_n$''.  Several other combinatorial objects are also enumerated by the Springer numbers: Weyl chambers in the principal Springer cone of the Coxeter group $B_n$ \cite{Springer1971}, topological types of odd functions with $2n$ critical values \cite{Arnold}, labeled ballot paths \cite{Chen-et-al}, and certain classes of complete binary trees and plane rooted forests \cite{J-V2014}. Also, Springer numbers are studied from the point of view of the classical moment problem \cite{Sokal2020}.

Based on experimental observations, Callan \cite{Callan} conjectured (and published his findings in  A001586 in \cite{oeis} in 2012) that the number of up-down permutations of even length fixed under reverse and complement is given by the Springer number. Examples of such permutations are 2413, 362514 and 57681324. In this paper, we prove this conjecture, thus providing yet another combinatorial interpretation of the Springer numbers. Moreover, we introduce the statistics LLE (the number of elements to the left of the left extreme elements) and BE (half of the number of elements between the extreme elements) that give two $q$-analogues and a $(p,q)$-analogue of the Springer numbers. We refer to \cite{EuFu} for a discussion of $q$- and $(p,q)$-analogues of the Springer numbers, realized on complete increasing binary trees where some leaves are unlabeled.

\subsection{Quadrant marked mesh patterns}

\noindent
{\em Quadrant marked mesh patterns} were introduced in \cite{kitrem}, but the paper \cite{KitRem2012} is most relevant in our context. These patterns are defined as follows.

Let $\pi = \pi_1 \ldots \pi_n$ be a permutation in  $S_n$. Then we will consider the
graph of $\pi$, $G(\pi)$, to be the set of points $(i,\pi_i)$ for
$i =1, \ldots, n$.  For example, the graph of the permutation
$\pi = 471569283$ is pictured in Figure~\ref{fig:basic}. Then if we draw a coordinate system centered at a
point $(i,\pi_i)$, we will be interested in  the points that
lie in the four quadrants I, II, III, and IV of that
coordinate system as pictured
in Figure \ref{fig:basic}.  For any $a,b,c,d \in
	\mathbb{N}= \{0,1,2, \ldots \}$,
we say that $\pi_i$ matches the
quadrant marked mesh pattern MMP$(a,b,c,d)$ in $\pi$ if in $G(\pi)$  relative
to the coordinate system which has the point $(i,\pi_i)$ as its
origin,  there are
$\geq a$ points in quadrant I,
$\geq b$ points in quadrant II, $\geq c$ points in quadrant
III, and $\geq d$ points in quadrant IV.
For example,
if $\pi = 471569283$, the point $\pi_4 =5$  matches
the quadrant marked mesh pattern MMP$(2,1,2,1)$ since relative
to the coordinate system with origin $(4,5)$,
there are 3 points in $G(\pi)$ in quadrant I,
1 point in $G(\pi)$ in quadrant II, 2 points in $G(\pi)$ in quadrant III, and 2 points in $G(\pi)$ in
quadrant IV.  Note that if a coordinate
in MMP$(a,b,c,d)$ is 0, then there is no condition imposed
on the points in the corresponding quadrant. We let $\mmp^{(a,b,c,d)}(\pi)$ be the number of occurrences of MMP$(a,b,c,d)$ in $\pi$.

\begin{figure}
	\begin{center}
		\includegraphics[scale=0.7]{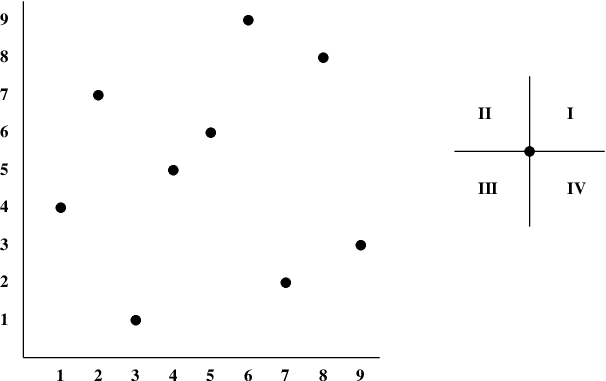}
		\caption{The graph of $\pi = 471569283$.}\label{fig:basic}
	\end{center}
\end{figure}

Note that by definition, occurrences of the patterns MMP$(1,0,0,0)$ (resp., MMP$(0,1,0,0)$, MMP$(0,0,1,0)$, MMP$(0,0,0,1)$) in $\pi$ are precisely occurrences of non-right-to-left maxima (resp., non-left-to-right maxima, non-left-to-right minima, non-right-to-left minima). This simple observation allows us, based on results in \cite{KitRem2012}, to derive distributions of  left-to-right (resp., right-to-left) maxima and minima on alternating permutations. Additionally, we derive a joint distribution of left-to-right and right-to-left maxima (resp., minima) on alternating permutations. To our knowledge, these distributions have not been recorded in the literature; only the distribution of left-to-right maxima on even-length permutations appears in A085734 in \cite{oeis} apparently as a result of computational experiments.  To derive the joint distributions, we generalize a result of Kitaev and Remmel (Theorem 1 in \cite{KitRem2012}) by finding joint distributions of (MMP(0,1,0,0), MMP(1,0,0,0)) and  (MMP(0,0,1,0), MMP(0,0,0,1)).  Hence, we refine classic enumeration results of Andr\'e~\cite{Andre1, Andre2} by providing new $q$-analogues and $(p,q)$-analogues for the number of alternating permutations.

Note that, by reverse and/or complement operations, all four statistics in question are equidistributed on alternating permutations. However, for technical reasons, we also need to consider the cases of even- and odd-length permutations separately. Consequently, the fact mentioned in A085734 in \cite{oeis} represents 1/24 of all distribution results presented in this paper. These include the 16 distributions for each of the four statistics on up-down and down-up permutations of even and odd lengths, as well as the 8 joint distributions of the statistics (lrmax, rlmax) (resp., (lrmin, rlmin)) on up-down and down-up permutations of even and odd lengths.

\subsection{Partially ordered patterns}
\noindent Kitaev~\cite{Kit2007} introduced the notion of a {\em partially ordered pattern} ({\em POP}), which has attracted significant attention in the literature; e.g., see~\cite{HHM-2022,YWZ}. In \cite{GaoKit2019} a more convenient way to define POPs is introduced, which we use in this paper.

A {\em partially ordered pattern} ({\em POP}) $p$ of length $k$ is a $k$-element partially ordered set (poset) $P$ labeled by the elements in $\{1,\ldots,k\}$. An occurrence of such a POP $p$ in a permutation $\pi=\pi_1\cdots\pi_n$ is a subsequence $\pi_{i_1}\cdots\pi_{i_k}$, where $1\leq i_1<\cdots< i_k\leq n$,  such that $\pi_{i_j}<\pi_{i_m}$ if $j<_P m$. Thus, a classical pattern of length $k$ corresponds to a $k$-element chain. For example, POP $p=$ \hspace{-3.5mm}
\begin{minipage}[c]{3.5em}\scalebox{1}{
		\begin{tikzpicture}[scale=0.5]

			\draw [line width=1](0,-0.5)--(0,0.5);

			\draw (0,-0.5) node [scale=0.4, circle, draw,fill=black]{};
			\draw (1,-0.5) node [scale=0.4, circle, draw,fill=black]{};
			\draw (0,0.5) node [scale=0.4, circle, draw,fill=black]{};

			\node [left] at (0,-0.6){${\small 3}$};
			\node [right] at (1,-0.6){${\small 2}$};
			\node [left] at (0,0.6){${\small 1}$};

		\end{tikzpicture}
	}\end{minipage}
occurs six times in the permutation 41523, namely, as the subsequences 412, 413, 452, 453, 423, and 523. Clearly, avoiding $p$ is the same as avoiding the patterns 312, 321	and 231, defined above, at the same time.

\begin{figure}[t]
	\centering

	\begin{tabular}{cccc}

		\begin{tikzpicture}[scale=0.6]

			\draw [line width=1](0,0)--(1.5,1.5);
			\draw [line width=1](1,0)--(1.5,1.5);
			\draw [line width=1](3,0)--(1.5,1.5);

			\draw (0,0) node [scale=0.4, circle, draw,fill=black]{};
			\draw (1,0) node [scale=0.4, circle, draw,fill=black]{};
			\draw (3,0) node [scale=0.4, circle, draw,fill=black]{};
			\draw (1.5,1.5) node [scale=0.4, circle, draw,fill=black]{};

			\draw (1.75,0) node [scale=0.15, circle, draw,fill=black]{};
			\draw (2,0) node [scale=0.15, circle, draw,fill=black]{};
			\draw (2.25,0) node [scale=0.15, circle, draw,fill=black]{};

			\node [below] at (0,-0.1){\footnotesize{$1$}};
			\node [below] at (1,-0.1){\footnotesize{$2$}};
			\node [below] at (3,-0.1){\footnotesize{$k-1$}};
			\node [above] at (1.5,1.5){\footnotesize{$k$}};
		\end{tikzpicture}

		 &

		\begin{tikzpicture}[scale=0.6]

			\draw [line width=1](0,0)--(1.5,1.5);
			\draw [line width=1](1,0)--(1.5,1.5);
			\draw [line width=1](3,0)--(1.5,1.5);

			\draw (0,0) node [scale=0.4, circle, draw,fill=black]{};
			\draw (1,0) node [scale=0.4, circle, draw,fill=black]{};
			\draw (3,0) node [scale=0.4, circle, draw,fill=black]{};
			\draw (1.5,1.5) node [scale=0.4, circle, draw,fill=black]{};

			\draw (1.75,0) node [scale=0.15, circle, draw,fill=black]{};
			\draw (2,0) node [scale=0.15, circle, draw,fill=black]{};
			\draw (2.25,0) node [scale=0.15, circle, draw,fill=black]{};

			\node [below] at (0,-0.1){\footnotesize{$2$}};
			\node [below] at (1,-0.1){\footnotesize{$3$}};
			\node [below] at (3,-0.1){\footnotesize{$k$}};
			\node [above] at (1.5,1.5){\footnotesize{$1$}};
		\end{tikzpicture}

		 &

		\begin{tikzpicture}[scale=0.6]

			\draw [line width=1](0,1.5)--(1.5,0);
			\draw [line width=1](1,1.5)--(1.5,0);
			\draw [line width=1](3,1.5)--(1.5,0);

			\draw (1.5,0) node [scale=0.4, circle, draw,fill=black]{};
			\draw (1,1.5) node [scale=0.4, circle, draw,fill=black]{};
			\draw (3,1.5) node [scale=0.4, circle, draw,fill=black]{};
			\draw (0,1.5) node [scale=0.4, circle, draw,fill=black]{};

			\draw (1.75,1.5) node [scale=0.15, circle, draw,fill=black]{};
			\draw (2,1.5) node [scale=0.15, circle, draw,fill=black]{};
			\draw (2.25,1.5) node [scale=0.15, circle, draw,fill=black]{};

			\node [above] at (0,1.5){\footnotesize{$1$}};
			\node [above] at (1,1.5){\footnotesize{$2$}};
			\node [above] at (3,1.5){\footnotesize{$k-1$}};
			\node [below] at (1.5,0){\footnotesize{$k$}};
		\end{tikzpicture}

		 &

		\begin{tikzpicture}[scale=0.6]

			\draw [line width=1](0,1.5)--(1.5,0);
			\draw [line width=1](1,1.5)--(1.5,0);
			\draw [line width=1](3,1.5)--(1.5,0);

			\draw (1.5,0) node [scale=0.4, circle, draw,fill=black]{};
			\draw (1,1.5) node [scale=0.4, circle, draw,fill=black]{};
			\draw (3,1.5) node [scale=0.4, circle, draw,fill=black]{};
			\draw (0,1.5) node [scale=0.4, circle, draw,fill=black]{};

			\draw (1.75,1.5) node [scale=0.15, circle, draw,fill=black]{};
			\draw (2,1.5) node [scale=0.15, circle, draw,fill=black]{};
			\draw (2.25,1.5) node [scale=0.15, circle, draw,fill=black]{};

			\node [above] at (0,1.5){\footnotesize{$2$}};
			\node [above] at (1,1.5){\footnotesize{$3$}};
			\node [above] at (3,1.5){\footnotesize{$k$}};
			\node [below] at (1.5,0){\footnotesize{$1$}};
		\end{tikzpicture}
	\end{tabular}
	\vspace{-3mm}
	\caption{Flat POPs of interest in our paper.}\label{pic-Bk}
\end{figure}
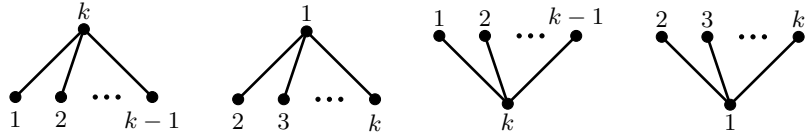

Of interest to us in this paper are POPs defined by \emph{flat posets}, examples of which are presented in Figure~\ref{pic-Bk}. We refer to these POPs as \emph{flat POPs}. Flat POPs were introduced in \cite{Kit2007}, and permutations avoiding any fixed flat POP were enumerated in \cite{GaoKit2019}. In this paper, we enumerate alternating permutations avoiding flat POPs in Figure~\ref{pic-Bk}.

In fact, for flat POPs in Figure~\ref{pic-Bk} one can easily extend the avoidance result in \cite{GaoKit2019} to obtain a recurrence relation for  the distribution of these patterns among all permutations. Indeed, suppose $P(n,\ell)$ denotes the number of $n$-permutations with $\ell$ occurrences of the  leftmost POP in Figure~\ref{pic-Bk} (any other POP in this figure is equivalent to this one by applying reverse and/or complement operations). Then, inserting the new largest element $(n+1)$ in an $n$-permutation, we have \\[-3mm]
\begin{equation}\label{POP-distr}P(n+1,\ell)=\sum_{j=1}^{n+1}P(n,\ell-{j-1\choose k-1})\end{equation}
where ${a\choose b}=0$ if $a<b$, $P(0,\ell)=1$ if $\ell=0$ and $P(0,\ell)=0$ if $\ell\neq 0$, and $P(n,0)$ is given in \cite[Thm 2]{GaoKit2019}.  To derive (\ref{POP-distr}) we used the following observation: if $(n+1)$ is inserted in position $j$, then ${j-1\choose k-1}$ new occurrences of the POP are introduced.

\subsection{Organization of the paper}

\noindent
The paper is organized as follows. In Section~\ref{Springer-sec}, we demonstrate that the number of up-down permutations fixed under reverse and complement is counted by the Springer numbers, thereby confirming Callan's conjecture and offering a new combinatorial interpretation of these numbers. Moreover, in Section~\ref{q-Springer-sec} (resp., Section~\ref{pq-Springer-sec}) we give two $q$-analogues (resp., a $(p,q)$-analogue) of the Springer numbers. In Section~\ref{dist-sec},  we find the distribution of every single minima/maxima statistic on alternating permutations of even and odd lengths, and record the result in Proposition~\ref{prop1} and Theorem~\ref{dist-single-pattern-thm}. In Section~\ref{joint-distr-mmp-sec},  we find the joint distribution of the statistics (MMP(0,1,0,0), MMP(1,0,0,0)) (resp., (MMP(0,0,1,0), MMP(0,0,0,1))) on up-down and down-up permutations of even and odd lengths. In Section~\ref{joint-distr-sec}, we find the joint distribution of the statistics $(\lrmax,\rlmax)$ (resp., $(\lrmin,\rlmin)$) on $UD_{2n}$, $UD_{2n-1}$, $DU_{2n}$ and $DU_{2n-1}$. In Section~\ref{POPs-sec},  we find the number of permutations in $UD_{2n}$, $UD_{2n+1}$, $DU_{2n}$ and $DU_{2n+1}$ that avoid any fixed POP in Figure~\ref{pic-Bk} for any $k\geq 3$.  Finally, in Section~\ref{Concluding-sec},  we provide concluding remarks and state open problems.

\section{A new combinatorial interpretation of the Springer numbers}\label{Springer-sec}
\noindent
In this section, we show that  the number of up-down permutations of even length fixed under reverse and complement is counted by the Springer numbers, which was conjectured by Callan~\cite{Callan}. Let $UD^{rc}_{2n}$ denote the set of these permutations of length $2n$. For $\pi=\pi_1\ldots\pi_{2n}\in UD^{rc}_{2n}$, we have
$$\pi_1\ldots\pi_{2n}=\pi_1\pi_2\cdots\pi_n(2n+1-\pi_n)\cdots (2n+1-\pi_2)(2n+1-\pi_1)$$
so that $\pi$ is uniquely determined by choosing $\pi_1\ldots\pi_n$. Note that if $\pi_i=2n$, then $i$ is even and the element  $1$ must be in odd position $2n+1-i$. Also, $\pi_1\ldots\pi_n$ either contains $2n$ or $1$, but not both. More generally, $\pi_1\ldots\pi_n$ contains exactly one element from the pair $\{i,2n+1-i\}$ for any $i$, $1\leq i\leq n$. But then, \\[-3mm]
\begin{equation}\label{rec-1} b_n=\sum_{k=0}^{n-1}2^k\binom{n-1}{k}E_kb_{n-k-1}\end{equation}
where $b_n$ is the number of permutations in $UD^{rc}_{2n}$, and $E_k$ is the number of up-down permutations of length $k$ discussed in Subsection~\ref{Euler-numbers-subsec}. Indeed, letting $k$ be the number of elements to the left of the element 1 or $2n$, whichever of them can be found in $\pi_1\ldots\pi_n$, $0\leq k\leq n-1$, we can choose $\pi_1\ldots\pi_k$ by
\begin{itemize}
	\item first selecting $k$ pairs of numbers  $\{i,2n+1-i\}$ in ${n-1\choose k}$ ways,
	\item then deciding, in $2^k$ ways, if the smaller or the larger element in each selected pair is to be used in $\pi_1\ldots \pi_k$,
	\item then forming $\pi_1\ldots\pi_k$ out of the chosen elements  in $E_k$ ways as it can be any up-down permutation (this will automatically fix the choice of $\pi_{2n-k+1}\pi_{2n-k+2}\ldots\pi_{2n}$), and finally
	\item observing that $\pi_{k+2}\pi_{k+3}\ldots\pi_{2n-k-1}$ (resp., $\pi_{2n-k-1}\pi_{2n-k-2}\ldots \pi_{k+2}$), in case $2n$ (resp., $1$) is in position $k+1$, can be any of the permutations counted by $b_{n-k-1}$ and formed from the elements not used in $\pi_1\ldots\pi_k$ and in $\pi_{2n-k+1}\pi_{2n-k+2}\ldots\pi_{2n}$.
\end{itemize}

Replacing $n$ with $n+1$ in (\ref{rec-1}), then multiplying both sides of the equation by $\frac{t^n}{n!}$ (the power of $t$ gives half-length of permutations in $UD^{rc}_{2n}$) and summing over all $n\geq 0$, we obtain \\[-2mm]
$$\sum_{n\geq 0}b_{n+1}\frac{t^{n}}{n!}=\sum_{n\geq 0}\sum_{k=0}^{n}2^k\binom{n}{k}E_kb_{n-k}\frac{t^{n}}{n!}, \mbox{ or equivalently,}$$
\vspace{-2mm}
\begin{equation}\label{useful-eq-1}
	\frac{\mbox{d}}{\mbox{d} t}\sum_{n\geq 0}b_{n+1}\frac{t^{n+1}}{(n+1)!}=\sum_{n\geq 0}\sum_{k=0}^{n}2^kE_k\frac{t^{k}}{k!}b_{n-k}\frac{t^{n-k}}{(n-k)!}.
\end{equation}
Letting $P(t)=\sum_{n\geq 0}b_n\frac{t^n}{n!}$, we obtain \\[-3mm]
\begin{equation}\label{diff-eq-1}
	\frac{\mbox{d}}{\mbox{d} t}P(t)=E(2t)P(t).
\end{equation}
By (\ref{E_n}),
$E(2t)=\sec 2t+\tan 2t$,  and since $P(0)=b_0=1$, we obtain the desired result that  \\[-3mm]
\begin{equation}\label{B(t)-formula}
	P(t)=\frac{1}{\cos t-\sin t}.
\end{equation}

\subsection{$q$-analogues of the Springer numbers}\label{q-Springer-sec}
\noindent
The {\em extreme elements} are the largest and  smallest elements in a permutation. For $\pi=\pi_1\ldots\pi_{2n}\in UD^{rc}_{2n}$, let  lle$(\pi)$ be the number of elements in $\pi$ strictly to the left of the left extreme element; the respective statistic is referred to as LLE. For example, lle$(17463528)=0$, lle$(28463517)=1$ and lle$(34172856)=2$. Letting the variable $q$ record the value of LLE, and following our steps in the derivation of (\ref{useful-eq-1}), this equation becomes
\begin{equation}\label{useful-eq-2}
	\frac{\partial}{\partial t}\sum_{n\geq 0}b_{n+1}(q)\frac{t^{n+1}}{(n+1)!}=\sum_{n\geq 0}\sum_{k=0}^{n}(2q)^kE_k\frac{t^{k}}{k!}b_{n-k}(1)\frac{t^{n-k}}{(n-k)!},
\end{equation}
where $b_n(q)=\sum_{\pi\in UD^{rc}_{2n}}q^{\mbox{lle}(\pi)}$ and $b_r(1) = b_r$ for  $r\geq 0$.
The ODE (\ref{diff-eq-1}) then becomes
$$\frac{\partial}{\partial t}Q(t,q)=E(2qt){\color{red}P}(t)$$
where $Q(t,q)=\sum_{n\geq 0}b_n(q)\frac{t^n}{n!}$ and $P(t)$ is given by (\ref{B(t)-formula}). Therefore,
\begin{equation}\label{B(t,q)-form}
	Q(t,q)={\color{red}1+}\int_{0}^{t}\frac{\sec 2qz + \tan 2qz}{\cos z - \sin z}dz.
\end{equation}
The initial terms in $Q(t,q)$, when expanded in $t$, are \\[-3mm]
$${\color{red}1+}t + (1+2q)\frac{t^2}{2!}+(3+4q+4q^2)\frac{t^3}{3!}+(11+18q+12q^2+16q^3)\frac{t^4}{4!}+$$
$$(57+88q+72q^2+64q^3+80q^4)\frac{t^5}{5!}+(361+570q+440q^2+480q^3+400q^4+512q^5)\frac{t^6}{6!}+\cdots.$$
Similarly, we next consider  the  statistic BE, defined as half the number of elements between the extreme elements in $UD^{rc}_{2n}$. We let be$(\pi)$ be the value of BE for $\pi\in UD^{rc}_{2n}$. For example, be$(47381625)=0$, be$(57163824)=1$ and be$(15372648)=3$. Letting the variable $p$ record the value of BE, and following our steps in the derivation of (\ref{useful-eq-1}), this equation becomes
\begin{equation}\label{useful-eq-3}
	\frac{\partial}{\partial t}\sum_{n\geq 0}c_{n+1}(p)\frac{t^{n+1}}{(n+1)!}=\sum_{n\geq 0}\sum_{k=0}^{n}2^kE_k\frac{t^{k}}{k!}b_{n-k}(1)\frac{(pt)^{n-k}}{(n-k)!},
\end{equation}
where $c_n(p)=\sum_{\pi\in UD^{rc}_{2n}}p^{\mbox{be}(\pi)}$. The ODE (\ref{diff-eq-1}) then becomes
$$\frac{\partial}{\partial t}U(t,p)=E(2t){\color{red}P}(pt)$$
where $U(t,p)=\sum_{n\geq 0}c_n(p)\frac{t^n}{n!}$, and therefore
\begin{equation}\label{U(t,q)-formula}
	U(t,p)={\color{red}1+}\int_{0}^{t}\frac{\sec 2z + \tan 2z}{\cos pz - \sin pz}dz.
\end{equation}
We note that by definition, for $n\geq 1$ and $\pi\in UD^{rc}_{2n}$, lle$(\pi)+$be$(\pi)=n-1$, which resembles the property of the classical statistics asc and des (the number of ascents and descents) on the set of all permutations. Hence, the coefficient of $q^it^{n}$ in $Q(t,q)$ is equal to that of $p^{n-1-i}t^{n}$ in $U(t,p)$.

\subsection{$(p,q)$-analogue of the Springer numbers}\label{pq-Springer-sec}
\noindent
We note that the equations (\ref{useful-eq-2}) and (\ref{useful-eq-3}) can be combined as
\begin{equation}\label{useful-eq-4}
	\frac{\partial}{\partial t}\sum_{n\geq 0}d_{n+1}(p,q)\frac{t^{n+1}}{(n+1)!}=\sum_{n\geq 0}\sum_{k=0}^{n}(2q)^kE_k\frac{t^{k}}{k!}{\color{red}d}_{n-k}(1,1)\frac{(pt)^{n-k}}{(n-k)!},
\end{equation}
where $d_n(p,q)=\sum_{\pi\in UD^{rc}_{2n}}p^{\mbox{be}(\pi)}q^{\mbox{lle}(\pi)}$ and ${\color{red}d}_{n-k}(1,1)=b_{n-k}$. The ODE (\ref{diff-eq-1}) then becomes
$$\frac{\partial}{\partial t}W(t,p,q)=E(2qt){\color{red}P}(pt)$$
where $W(t,p,q)=\sum_{n\geq 0}d_n(p,q)\frac{t^n}{n!}$, and therefore
\begin{equation}\label{W(t,q)-formula}
	W(t,p,q)={\color{red}1+}\int_{0}^{t}\frac{\sec 2qz + \tan 2qz}{\cos pz - \sin pz}dz
\end{equation}
which gives a $(p,q)$-analogue of the Springer numbers.

\section{Distribution of a single minima/maxima statistic}\label{dist-sec}
\noindent
In this section, we determine the distributions of left-to-right and right-to-left minima and maxima statistics for up-down and down-up permutations of both even and odd lengths (a total of 16 distributions). In Proposition~\ref{prop1}, we classify the distributions into four classes, and in Theorem~\ref{dist-single-pattern-thm}, we present formulae for their corresponding exponential generating functions.

For $n\geq 1$, we let \\[-0.6cm]
\begin{center}
	\begin{tabular}{ccc}
		$F^{(1)}_{2n}(q)=\sum_{\pi \in UD_{2n}}q^{\rlmax(\pi)}$ & \ \  & $F^{(2)}_{2n-1}(q)=\sum_{\pi \in UD_{2n-1}}q^{\rlmax(\pi)}$ \\[3mm]
		$F^{(3)}_{2n}(q)=\sum_{\pi \in DU_{2n}}q^{\rlmax(\pi)}$ &      & $F^{(4)}_{2n-1}(q)=\sum_{\pi \in DU_{2n-1}}q^{\rlmax(\pi)}$
	\end{tabular}
\end{center}

\begin{prop}\label{prop1}
	For all $n\geq 1$, \\[-5mm]
	\begin{itemize}
		\item[{\upshape 1.}] $	F^{(1)}_{2n}(q)
			      =\sum_{\pi \in DU_{2n}}q^{\lrmax(\pi)}
			      =\sum_{\pi \in DU_{2n}}q^{\rlmin(\pi)}
			      =\sum_{\pi \in UD_{2n}}q^{\lrmin(\pi)}$.
		\item[{\upshape 2.}] $	F^{(2)}_{2n-1}(q)
			      =\sum_{\pi \in UD_{2n-1}}q^{\lrmax(\pi)}
			      =\sum_{\pi \in DU_{2n-1}}q^{\rlmin(\pi)}
			      =\sum_{\pi \in DU_{2n-1}}q^{\lrmin(\pi)}$.
		\item[{\upshape 3.}] $	F^{(3)}_{2n}(q)
			      =\sum_{\pi \in UD_{2n}}q^{\lrmax(\pi)}
			      =\sum_{\pi \in UD_{2n}}q^{\rlmin(\pi)}
			      =\sum_{\pi \in DU_{2n}}q^{\lrmin(\pi)}$.
		\item[{\upshape 4.}] $	F^{(4)}_{2n-1}(q)
			      =\sum_{\pi \in DU_{2n-1}}q^{\lrmax(\pi)}
			      =\sum_{\pi \in UD_{2n-1}}q^{\rlmin(\pi)}
			      =\sum_{\pi \in UD_{2n-1}}q^{\lrmin(\pi)}$.
	\end{itemize}
\end{prop}

\begin{proof}
	For any permutation $\pi \in S_{n}$, it is clear that
	$$\rlmax(\pi)=\lrmax(\pi^r)=\rlmin(\pi^c)=\lrmin(\pi^{rc}).$$
	This observation gives part 1, since
	$$\pi \in UD_{2n}\Longleftrightarrow\pi^r \in DU_{2n}\Longleftrightarrow\pi^c \in DU_{2n}\Longleftrightarrow\pi^{rc} \in UD_{2n}.$$
	The proofs for parts 2, 3, and 4 follow similar lines, and hence are omitted.
\end{proof}

By Proposition~\ref{prop1}, the study of the distributions of  $\lrmax$, $\rlmax$, $\lrmin$ and $\rlmin$ in the sets of up-down and down-up permutations of even and odd lengths can be reduced to the study of the following four generating functions: \\[-5mm]
\begin{center}
	\begin{tabular}{ccc}
		$F^{(1)}(t, q)=1+\sum_{n\geq 1}F^{(1)}_{2n}(q)\frac{t^{2n}}{(2n)!}$ & \ \  &
		$F^{(2)}(t, q)=\sum_{n\geq 1}F^{(2)}_{2n-1}(q)\frac{t^{2n-1}}{(2n-1)!}$      \\[3mm]
		$F^{(3)}(t, q)=1+\sum_{n\geq 1}F^{(3)}_{2n}(q)\frac{t^{2n}}{(2n)!}$ &      &
		$F^{(4)}(t, q)=\sum_{n\geq 1}F^{(4)}_{2n-1}(q)\frac{t^{2n-1}}{(2n-1)!}$
	\end{tabular}
\end{center}

We use Theorem~1 in \cite{KitRem2012} to prove the following theorem.

\begin{thm}\label{dist-single-pattern-thm}
	We have \\[-5mm]
	\begin{eqnarray*}
		F^{(1)}(t, q)&=&(\sec (t))^{q},\\
		F^{(2)}(t, q)&=&(\sec (t))^{q}\int_{0}^{qt}(\sec {\color{red}(}z/q{\color{red})})^{-q}dz,\\
		F^{(3)}(t, q)&=&1+\int_{0}^{qt}(\sec {\color{red}(}y/q{\color{red})})^{1+q}\int_{0}^{y}(\sec {\color{red}(}z/q{\color{red})})^{-q}dzdy, \\
		F^{(4)}(t, q)&=&\int_{0}^{qt}(\sec {\color{red}(}z/q{\color{red})})^{1+q}dz.
	\end{eqnarray*}
\end{thm}
\begin{proof}
	For any $\pi \in S_n$, we have
	\begin{align}\label{eq:mmp}
		\mmp^{(1,0,0,0)}(\pi)=n-\rlmax(\pi).
	\end{align}
	The distributions of the statistic MMP$(1,0,0,0)$  on alternating permutations of even and odd lengths are given in Theorem~1 in \cite{KitRem2012}. In particular, the distribution of MMP$(1,0,0,0)$ on up-down permutations of even length is
	\begin{align*}
		(\sec (qt))^{1/q} & = 1+\sum_{n\geq 1}\sum_{\pi\in UD_{2n}}q^{\mmp^{(1,0,0,0)}(\pi)}\frac{t^{2n}}{(2n)!}
		=1+\sum_{n\geq1}^{}\sum_{\pi\in UD_{2n}}q^{2n-\rlmax(\pi)}\frac{t^{2n}}{(2n)!}                             \\
		                  & =1+\sum_{n\geq 1}^{}\sum_{\pi\in UD_{2n}}(q^{-1})^{\rlmax(\pi)}\frac{(qt)^{2n}}{(2n)!}
		=F^{(1)}(qt, q^{-1})
	\end{align*}
	Hence, $F^{(1)}(t, q)=(\sec (t))^{q}$. 	By similar arguments (and using Theorem~1\footnote{In Theorem~1 of \cite{KitRem2012}, the formula for $C(t,x)$ requires the factor $(\sec(xz))^{-1/x}$ rather than $(\sec(xz))^{1/x}$. This typo in \cite{KitRem2012} led us to omit the minus sign in the second integral in the formula for $F^{(3)}(t,q)$ in Theorem~\ref{dist-single-pattern-thm} of the published version of this paper. The omission was pointed out by Jesse Geneson.} in \cite{KitRem2012}), we obtain the formulae for $F^{(2)}(t, q)$, $F^{(3)}(t, q)$ and $F^{(4)}(t, q)$.
\end{proof}

\begin{rem} The formula $F^{(1)}(t, q)=(\sec (t))^{q}$ was already derived in $1975$ by Carlitz and  Scoville \cite{CarlitzScoville}. Moreover, the formula $F^{(2)}(t, q)$ is essentially the formula (3.4) in \cite{CarlitzScoville}, and the formula $F^{(4)}(t, q)$ can be obtained from the PDE (4.5) in  \cite{CarlitzScoville}. \end{rem}

\section{Joint distributions of MMP statistics}\label{joint-distr-mmp-sec}
\noindent
In this section, we find the joint distribution of the statistics (MMP(0,1,0,0), MMP(1,0,0,0)) (resp., (MMP(0,0,1,0), MMP(0,0,0,1))) on up-down and down-up permutations of even and odd lengths (8 distributions in total). In Proposition~\ref{prop2}, we classify the distributions into four classes, and in Theorem~\ref{joint-dist-mmp-thm}, we present formulae for their corresponding exponential generating functions.

For $n \geq 1$, we let\\[-0.5cm]
\begin{footnotesize}
	\begin{center}
		\begin{tabular}{ccc}
			$A_{2n}(p,q)=\sum_{\pi\in UD_{2n}}p^{\mmp^{(0,1,0,0)}(\pi)}q^{\mmp^{(1,0,0,0)}(\pi)} $ &  & $B_{2n-1}(p,q)=\sum_{\pi\in UD_{2n-1}}p^{\mmp^{(0,1,0,0)}(\pi)}q^{\mmp^{(1,0,0,0)}(\pi)}$ \\[3mm]
			$C_{2n}(p,q)=\sum_{\pi\in DU_{2n}}p^{\mmp^{(0,1,0,0)}(\pi)}q^{\mmp^{(1,0,0,0)}(\pi)}$  &  & $D_{2n-1}(p,q)=\sum_{\pi\in DU_{2n-1}}p^{\mmp^{(0,1,0,0)}(\pi)}q^{\mmp^{(1,0,0,0)}(\pi)}$
		\end{tabular}
	\end{center}
\end{footnotesize}
Based on the relationship between quadrant marked mesh patterns and maxima or minima statistics, as exemplified by~\eqref{eq:mmp}, the following proposition can be derived similarly to the proof of Proposition~\ref{prop1}.
\begin{prop}\label{prop2}
	For all $n\geq 1$,
	\begin{itemize}
		\item[{\upshape 1.}] $A_{2n}(p,q)
			      =\sum\limits_{\pi \in DU_{2n}}p^{\mmp^{(1,0,0,0)}(\pi)}q^{\mmp^{(0,1,0,0)}(\pi)}
			      =\sum\limits_{\pi \in DU_{2n}}p^{\mmp^{(0,0,1,0)}(\pi)}q^{\mmp^{(0,0,0,1)}(\pi)}
			      \\=\sum\limits_{\pi \in UD_{2n}}p^{\mmp^{(0,0,0,1)}(\pi)}q^{\mmp^{(0,0,1,0)}(\pi)}$.
		\item[{\upshape 2.}] $	B_{2n-1}(p,q)
			      =\sum\limits_{\pi \in UD_{2n-1}}p^{\mmp^{(1,0,0,0)}(\pi)}q^{\mmp^{(0,1,0,0)}(\pi)}
			      =\sum\limits_{\pi \in DU_{2n-1}}p^{\mmp^{(0,0,1,0)}(\pi)}q^{\mmp^{(0,0,0,1)}(\pi)}
			      \\=\sum\limits_{\pi \in DU_{2n-1}}p^{\mmp^{(0,0,0,1)}(\pi)}q^{\mmp^{(0,0,1,0)}(\pi)}$.
		\item[{\upshape 3.}] $	C_{2n}(p,q)
			      =\sum\limits_{\pi \in UD_{2n}}p^{\mmp^{(1,0,0,0)}(\pi)}q^{\mmp^{(0,1,0,0)}(\pi)}
			      =\sum\limits_{\pi \in UD_{2n}}p^{\mmp^{(0,0,1,0)}(\pi)}q^{\mmp^{(0,0,0,1)}(\pi)}
			      \\=\sum\limits_{\pi \in DU_{2n}}p^{\mmp^{(0,0,0,1)}(\pi)}q^{\mmp^{(0,0,1,0)}(\pi)}$.
		\item[{\upshape 4.}] $	D_{2n-1}(p,q)
			      =\sum\limits_{\pi \in DU_{2n-1}}p^{\mmp^{(1,0,0,0)}(\pi)}q^{\mmp^{(0,1,0,0)}(\pi)}
			      =\sum\limits_{\pi \in UD_{2n-1}}p^{\mmp^{(0,0,1,0)}(\pi)}q^{\mmp^{(0,0,0,1)}(\pi)}
			      \\=\sum\limits_{\pi \in UD_{2n-1}}p^{\mmp^{(0,0,0,1)}(\pi)}q^{\mmp^{(0,0,1,0)}(\pi)}$.
	\end{itemize}
\end{prop}

By Proposition~\ref{prop2}, the study of joint distributions of  (MMP(0,1,0,0), MMP(1,0,0,0)) and  (MMP(0,0,1,0), MMP(0,0,0,1)) in the sets $UD_{2n}$, $UD_{2n-1}$, $DU_{2n}$ and $DU_{2n-1}$ can be reduced to the study of the following four generating functions:\\[-7mm]
\begin{center}
	\begin{tabular}{ccc}
		$A(t,p,q)=1+\sum_{n\geq 1}A_{2n}(p,q)\frac{t^{2n}}{(2n)!}$ & \ \  &
		$B(t,p,q)=\sum_{n\geq 1}B_{2n{\color{red}-1}}(p,q)\frac{t^{2n{\color{red}-1}}}{(2n{\color{red}-1})!}$            \\[3mm]
		$C(t,p,q)=1+\sum_{n\geq 1}C_{2n}(p,q)\frac{t^{2n}}{(2n)!}$ &      &
		$D(t,p,q)=\sum_{n\geq 1}D_{2n{\color{red}-1}}(p,q)\frac{t^{2n{\color{red}-1}}}{(2n{\color{red}-1})!}$
	\end{tabular}
\end{center}

The formulae for $A(t,p,q)$, $B(t,p,q)$, $C(t,p,q)$ and $D(t,p,q)$ are derived in Sections~\ref{A(t,p,q)-sec}--\ref{D(t,p,q)-sec}, respectively, and we summarize them in the following theorem.

\begin{thm}\label{joint-dist-mmp-thm}
	We have \\[-5mm]
	\begin{eqnarray*}
		A(t,p,q)&=&{\color{red}1+}\int_{0}^{t}\left[(\sec (pqs))^{\frac{1}{p}+\frac{1}{q}}\int_{0}^{qs}(\sec (pz))^{-\frac{1}{p}}dz\right]ds,\\
		B(t,p,q)&=&t+\int_{0}^{t}\left[(\sec (pqs))^{\frac{1}{p}+\frac{1}{q}}\int_{0}^{qs}(\sec (pz))^{-\frac{1}{p}}dz   \int_{0}^{ps}(\sec (qz))^{-\frac{1}{q}}dz\right]ds,\\
		C(t,p,q)&=&{\color{red}1+}\int_{0}^{t}\left[(\sec (pqs))^{\frac{1}{p}+\frac{1}{q}}\int_{0}^{ps}(\sec (qz))^{-\frac{1}{q}}dz\right]ds, \\
		D(t,p,q)&=&\int_{0}^{t}(\sec (pqz))^{\frac{1}{p}+\frac{1}{q}}dz.
	\end{eqnarray*}
\end{thm}

\subsection{The generating function $A(t,p,q)$}\label{A(t,p,q)-sec}
\noindent If $\pi=\pi_1\cdots\pi_{2n}\in UD_{2n}$, then $2n$ must occur in one of the positions $2, 4,\ldots , 2n$. Let $UD_{2n}^{(2k)}$
denote the set of permutations $\pi\in UD_{2n}$ such that $\pi_{2k}=2n$. A schematic diagram of a permutation
in $UD_{2n}^{(2k)}$ is pictured in Figure~\ref{fig:ud2k2n}.
\begin{figure}
	\begin{center}
		\includegraphics[scale=0.6]{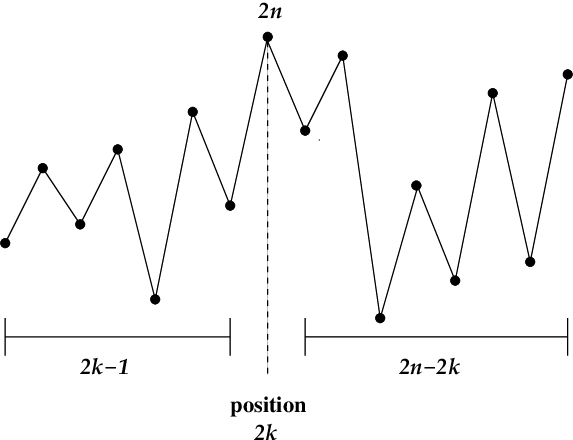}
	\end{center}
	\vspace{-5mm}
	\caption{The graph of a permutation $\pi \in UD_{2n}^{(2k)}$.}\label{fig:ud2k2n}
\end{figure}

Note that there are $\binom{2n-1}{2k-1}$
ways to pick the elements which occur to the left of position
$2k$ in such $\pi$. These elements form a permutation in $UD_{2k-1}$, and each of them contributes to the statistic MMP(1,0,0,0). Thus the contribution of the elements to the left of position $2k$ in $\sum_{\pi\in UD_{2n}^{(2k)}}p^{\mmp^{(0,1,0,0)}(\pi)}q^{\mmp^{(1,0,0,0)}(\pi)}$ is $q^{2k-1}B_{2k-1}(p,1)$.
The elements to the right of position $2k$  form a
permutation in $UD_{2n-2k}$, and each of these elements contributes to MMP(0,1,0,0). Since the elements to the left of position $2k$ have no effect
on whether an element to the right of position $2k$ contributes to MMP(1,0,0,0), and the elements to the right of position $2k$ have no effect
on whether an element to the left of position $2k$ contributes to MMP(0,1,0,0),
it follows
that the contribution of the elements to the right of position $2k$ in$\sum_{\pi\in UD_{2n}^{(2k)}}p^{\mmp^{(0,1,0,0)}(\pi)}q^{\mmp^{(1,0,0,0)}(\pi)}$ is $p^{2n-2k}A_{2n-2k}(1,q)$.
It thus follows that \\[-3mm]
$$A_{2n}(p,q)=\sum_{k=1}^{n}\binom{2n-1}{2k-1}q^{2k-1}B_{2k-1}(p,1)p^{2n-2k}A_{2n-2k}(1,q)$$
or, equivalently,
\begin{equation}\label{eq1}
	\frac{A_{2n}(p,q)}{(2n-1)!}=\sum_{k=1}^{n}\frac{q^{2k-1}B_{2k-1}(p,1)}{(2k-1)!}\frac{p^{2n-2k}A_{2n-2k}(1,q)}{(2n-2k)!}.
\end{equation}
Multiplying both sides of (\ref{eq1}) by $t^{2n-1}$ and summing for $n \geq 1$ , we see that
\begin{align*}
	\sum_{n \geq 1}\frac{A_{2n}(p,q)t^{2n-1}}{(2n-1)!} & =\left(\sum_{n \geq 1}\frac{(qt)^{2n-1}B_{2n-1}(p,1)}{(2n-1)!}\right)\left(\sum_{n \geq 0}\frac{(pt)^{2n}A_{2n}(1,q)}{(2n)!}\right) \\
	                                                   & =B(qt,p,1)A(pt,1,q)
\end{align*}
so that \\[-5mm]
\begin{equation}\label{A=BA-equat-1}\frac{\partial}{\partial t}{A(t,p,q)}=B(qt,p,1)A(pt,1,q)\end{equation}
with  initial condition $A(0,p,q) = 1$. By Proposition 1 and Theorem 1 in \cite{KitRem2012}, \\[-2mm]
$$A(t,1,q)=(\sec (qt))^{\frac{1}{q}} \mbox{\ \ and\ \  } B(t,p,1)=(\sec (pt))^{\frac{1}{p}}\int_{0}^{t}(\sec (pz))^{-\frac{1}{p}}dz.$$
The solution to (\ref{A=BA-equat-1}) is \\[-5mm]
$$A(t,p,q)={\color{red}1+}\int_{0}^{t}\left[(\sec (pqs))^{\frac{1}{p}+\frac{1}{q}}\int_{0}^{qs}(\sec (pz))^{-\frac{1}{p}}dz\right]ds.$$

\subsection{The generating function $B(t,p,q)$}\label{B(t,p,q)-sec}

\noindent
If $\pi=\pi_1\cdots\pi_{2n+1}\in UD_{2n+1}$, then $2n+1$ must occur in one of the positions $2, 4,\ldots , 2n$. Let $UD_{2n+1}^{(2k)}$
denote the set of permutations $\pi\in UD_{2n+1}$ such that $\pi_{2k}=2n+1.$ A schematic diagram of a permutation in  $UD_{2n+1}^{(2k)}$ is shown in Figure~\ref{fig:ud2k2n+1}.
\begin{figure}
	\begin{center}
		\includegraphics[scale=0.6]{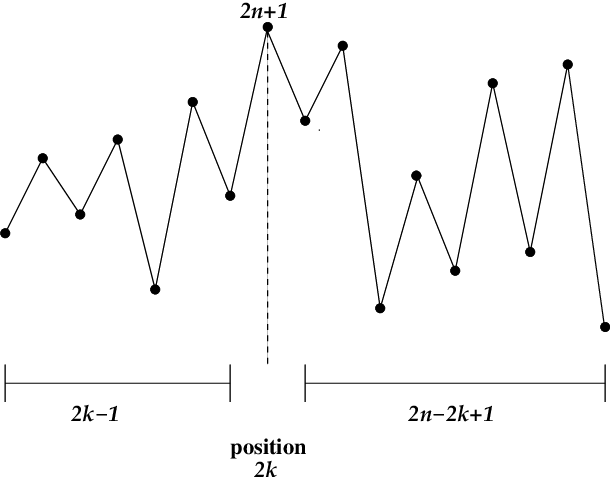}
	\end{center}
	\vspace{-5mm}
	\caption{The graph of a permutation $\pi \in UD_{2n+1}^{(2k)}$.}\label{fig:ud2k2n+1}
\end{figure}

There are $\binom{2n}{2k-1}$
ways to pick the elements occurring to the left of position
$2k$ in such $\pi$. These elements form a permutation in $UD_{2k-1}$, and each of them contributes to MMP(1,0,0,0). Thus the contribution of the elements to the left of position $2k$ in $\sum_{\pi\in UD_{2n+1}^{(2k)}}p^{\mmp^{(0,1,0,0)}(\pi)}q^{\mmp^{(1,0,0,0)}(\pi)}$ is $q^{2k-1}B_{2k-1}(p,1)$.
The elements to the right of position $2k$  form a
permutation in $UD_{2n-2k+1}$, and each of these elements contributes to MMP(0,1,0,0). Since the elements to the left (resp., right) of position $2k$ have no effect
on whether an element to the right (resp., left) of position $2k$ contributes to MMP(1,0,0,0) (resp., MMP(0,1,0,0)), it follows that the contribution of the elements to the right of position $2k$ in $\sum_{\pi\in UD_{2n{\color{red}+1}}^{(2k)}}p^{\mmp^{(0,1,0,0)}(\pi)}q^{\mmp^{(1,0,0,0)}(\pi)}$ is $p^{2n-2k+1}B_{2n-2k+1}(1,q)$.
It thus follows that for $n\geq 1$,
$$B_{2n+1}(p,q)=\sum_{k=1}^{n}\binom{2n}{2k-1}q^{2k-1}B_{2k-1}(p,1){\color{red}p^{2n-2k+1}B_{2n-2k+1}(1,q)}.$$
Hence for $n\geq 1$,
\begin{equation}\label{eq2}
	\frac{B_{2n+1}(p,q)}{(2n)!}=\sum_{k=1}^{n}\frac{q^{2k-1}B_{2k-1}(p,1)}{(2k-1)!}\frac{p^{2n-2k+1}B_{2n-2k+1}(1,q)}{(2n-2k+1)!}.
\end{equation}
Multiplying both sides of (\ref{eq2}) by $t^{2n}$, summing for $n \geq 1$, and considering that $B_1(p,1)=1$, we see that
\begin{align*}
	\sum_{n \geq 0}\frac{B_{2n+1}(p,q)t^{2n}}{(2n)!} & =1+\left(\sum_{n \geq 0}\frac{(qt)^{2n+1}B_{2n+1}(p,1)}{(2n+1)!}\right)\left(\sum_{n \geq 0}\frac{(pt)^{2n+1}B_{2n+1}(1,q)}{(2n+1)!}\right) \\
	                                                 & ={\color{red}1+}B(qt,p,1)B(pt,1,q)
\end{align*}
so that
\begin{equation}\label{diff-B}
	\frac{\partial}{\partial t}{B(t,p,q)}=1+B(qt,p,1)B(pt,1,q)
\end{equation}
with initial condition  $B(0,p,q) = 0$.
By Proposition 1 and Theorem 1 in \cite{KitRem2012},
$$ B(t,p,1)=(\sec (pt))^{\frac{1}{p}}\int_{0}^{t}(\sec (pz))^{-\frac{1}{p}}dz \mbox{\ \ and\ \ }B(t,1,q)=(\sec (qt))^{\frac{1}{q}}\int_{0}^{t}(\sec (qz))^{-\frac{1}{q}}dz.$$
The solution to (\ref{diff-B}) is then
$$B(t,p,q)=t+\int_{0}^{t}\left[(\sec (pqs))^{\frac{1}{p}+\frac{1}{q}}\int_{0}^{qs}(\sec (pz))^{-\frac{1}{p}}dz   \int_{0}^{ps}(\sec (qz))^{-\frac{1}{q}}dz\right]ds$$

\subsection{The generating function $C(t,p,q)$}\label{C(t,p,q)-sec}
\noindent
If $\pi=\pi_1\cdots\pi_{2n}\in DU_{2n}$, then $2n$ must occur in one of the positions $1, 3,\ldots , 2n-1$. Let $DU_{2n}^{(2k+1)}$
denote the set of permutations $\pi\in DU_{2n}$ such that $\pi_{2k+1}=2n.$ A schematic diagram of a permutation in $DU_{2n}^{(2k+1)}$
is in Figure~\ref{fig:du2k2n}.
\begin{figure}
	\begin{center}
		\includegraphics[scale=0.6]{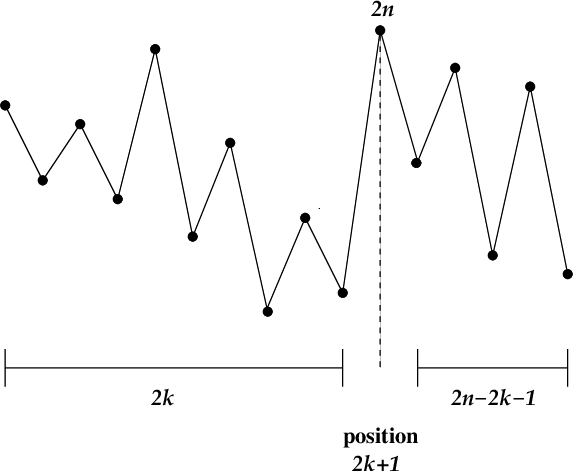}
	\end{center}
	\vspace{-5mm}
	\caption{The graph of a permutation $\pi \in DU_{2n}^{(2k+1)}$}.\label{fig:du2k2n}
\end{figure}

Note that there are $\binom{2n-1}{2k}$
ways to pick the elements occurring to the left of position
$2k+1$ in such $\pi$. These elements  form a  permutation in $DU_{2k}$, and each of them contributes to MMP(1,0,0,0). Thus the contribution of the elements to the left of position $2k+1$ in $\sum_{\pi\in DU_{2n}^{(2k+1)}}p^{\mmp^{(0,1,0,0)}(\pi)}q^{\mmp^{(1,0,0,0)}(\pi)}$ is $q^{2k}{\color{red}C}_{2k}(p,1)$.
The elements to the right of position $2k+1$  form a
permutation in $UD_{2n-2k-1}$, and each of these elements contributes to MMP(0,1,0,0). Since the elements to the left of position $2k+1$ have no effect
on whether an element to the right of position $2k+1$ contributes to MMP(1,0,0,0), and the elements to the right of position $2k+1$ have no effect on whether an element to the left of position $2k+1$ contributes to MMP(0,1,0,0),
it follows
that the contribution of the elements to the right of position $2k+1$ in $\sum_{\pi\in DU_{2n}^{(2k+1)}}p^{\mmp^{(0,1,0,0)}(\pi)}q^{\mmp^{(1,0,0,0)}(\pi)}$ is $p^{2n-2k-1}B_{2n-2k-1}(1,q)$.
It thus follows that
$$C_{2n}(p,q)=\sum_{k=0}^{n-1}\binom{2n-1}{2k}q^{2k}{\color{red}C}_{2k}(p,1)p^{2n-2k-1}B_{2n-2k-1}(1,q)$$
or, equivalently,
\begin{equation}\label{eq3}
	\frac{C_{2n}(p,q)}{(2n-1)!}=\sum_{k=0}^{n-1}\frac{q^{2k}{\color{red}C}_{2k}(p,1)}{(2k)!}\frac{p^{2n-2k-1}B_{2n-2k-1}(1,q)}{(2n-2k-1)!}.
\end{equation}
Multiplying both sides of (\ref{eq3}) by $t^{2n-1}$ and summing for $n \geq 1$ , we see that
\begin{align*}
	\sum_{n \geq 1}\frac{C_{2n}(p,q)t^{2n-1}}{(2n-1)!} & =\left(\sum_{n \geq 0}\frac{(qt)^{2n}{\color{red}C}_{2n}(p,1)}{(2n)!}\right)\left(\sum_{n \geq 1}\frac{(pt)^{2n-1}B_{2n-1}(1,q)}{(2n-1)!}\right) \\
	                                                   & ={\color{red}C}(qt,p,1)B(pt,1,q).
\end{align*}
So that
$$\frac{\partial}{\partial t}{C(t,p,q)}={\color{red}C}(qt,p,1)B(pt,1,q)$$
with initial condition $C(0,p,q) = 1$.
By Proposition 1 and Theorem 1 in \cite{KitRem2012},
$$C(t,p,1)=(\sec (pt))^{\frac{1}{p}} \mbox{\ \ and\ \ } B(t,1,q)=(\sec (qt))^{\frac{1}{q}}\int_{0}^{t}(\sec (qz))^{-\frac{1}{q}}dz.$$
Hence, we have \\[-3mm]
$$C(t,p,q)={\color{red}1+}\int_{0}^{t}\left[(\sec (pqs))^{\frac{1}{p}+\frac{1}{q}}\int_{0}^{ps}(\sec (qz))^{-\frac{1}{q}}dz\right]ds.$$

\subsection{The generating function $D(t,p,q)$}\label{D(t,p,q)-sec}

\noindent
If $\pi=\pi_1\cdots\pi_{2n+1}\in DU_{2n+1}$, then $2n+1$ must occur in one of the positions $1, 3,\ldots , 2n+1$. Let $DU_{2n+1}^{(2k+1)}$
denote the set of permutations $\pi\in DU_{2n+1}$ such that $\pi_{2k+1}=2n+1.$ A schematic diagram of a permutation in $DU_{2n+1}^{(2k+1)}$
is pictured in Figure~\ref{fig:du2k2n+1}.
\begin{figure}
	\begin{center}
		\includegraphics[scale=0.6]{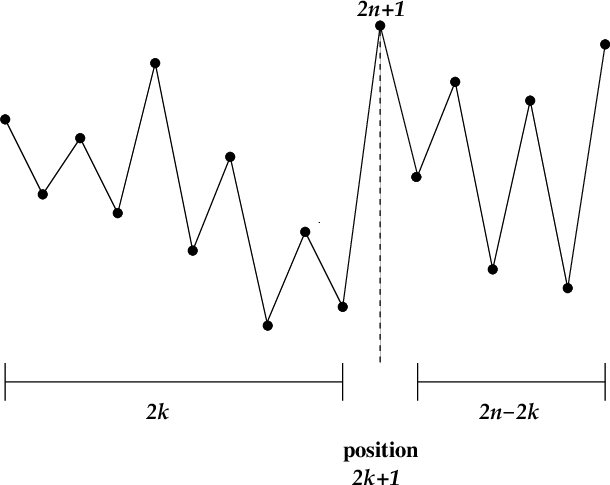}
	\end{center}
	\vspace{-5mm}
	\caption{The graph of a permutation $\pi \in DU_{2n+1}^{(2k+1)}$.}\label{fig:du2k2n+1}
\end{figure}

Note that there are $\binom{2n}{2k}$
ways to pick the elements which occur to the left of position
$2k+1$ in such $\pi$. These elements form a permutation in $DU_{2k}$, and each of them  contributes to MMP(1,0,0,0). Thus the contribution of the elements to the left of position $2k+1$ in $\sum_{\pi\in DU_{2n+1}^{(2k+1)}}p^{\mmp^{(0,1,0,0)}(\pi)}q^{\mmp^{(1,0,0,0)}(\pi)}$ is $q^{2k}C_{2k}(p,1)$.
The elements to the right of position $2k+1$  form a
permutation in $UD_{2n-2k}$. Each of these elements contributes to MMP(0,1,0,0). Since the elements to the left (resp., right) of position $2k+1$ have no effect
on whether an element to the right (resp., left) of position $2k+1$ contributes to MMP(1,0,0,0) (resp., MMP(0,1,0,0)),
it follows
that the contribution of the elements to the right of position $2k+1$ in $\sum_{\pi\in DU_{2n+1}^{(2k+1)}}p^{\mmp^{(0,1,0,0)}(\pi)}q^{\mmp^{(1,0,0,0)}(\pi)}$ is $p^{2n-2k}A_{2n-2k}(1,q)$. Therefore,
$$D_{2n+1}(p,q)=\sum_{k=0}^{n}\binom{2n}{2k}{\color{red}q^{2k}C_{2k}(p,1)}p^{2n-2k}A_{2n-2k}(1,q).$$
Hence, for ${\color{red}n\geq 0}$,
\begin{equation}\label{eq4}
	\frac{D_{2n+1}(p,q)}{(2n)!}=\sum_{k=0}^{n}\frac{{\color{red}q^{2k}C_{2k}(p,1)}}{(2k)!}\frac{p^{2n-2k}A_{2n-2k}(1,q)}{(2n-2k)!}.
\end{equation}
Multiplying both sides of (\ref{eq4}) by $t^{2n}$ and summing for $n \geq 0$ , we see that
\begin{align*}
	\sum_{n \geq 0}\frac{D_{2n+1}(p,q)t^{2n}}{(2n)!} & ={\color{red}\left(\sum_{n \geq 0}\frac{(qt)^{2n}C_{2n}(p,1)}{(2n)!}\right)}\left(\sum_{n \geq 0}\frac{(pt)^{2n}A_{2n}(1,q)}{(2n)!}\right)
	={\color{red}C(qt,p,1)}A(pt,1,q)
\end{align*}
so that
\begin{equation*}\label{diff-D}
	\frac{\partial}{\partial t}{D(t,p,q)}={\color{red}C(qt,p,1)}A(pt,1,q)
\end{equation*}
with initial the condition $D(0, p,q) = 0$.
By Proposition 1 and Theorem 1 in \cite{KitRem2012}, $A(t,1,q)=(\sec (qt))^{\frac{1}{q}}$.
Hence, we have
$$D(t,p,q)=\int_{0}^{t}(\sec (pqz))^{\frac{1}{p}+\frac{1}{q}}dz.$$

\section{Joint distributions of the maxima or minima statistics}\label{joint-distr-sec}
\noindent
In this section, we find the joint distribution of the statistics $(\lrmax,\rlmax)$ (resp., $(\lrmin,\rlmin)$) on up-down and down-up permutations of even and odd lengths (8 distributions in total). This generalizes our results in Section~\ref{dist-sec}.
In Proposition~\ref{prop3}, we classify the distributions into four classes, and in Theorem~\ref{joint-dist-pattern-thm}, we present formulae for their corresponding exponential generating functions. We note that our studies extend the scope of the respective results of Carlitz and Scoville in \cite{CarlitzScoville}, who obtain them in terms of certain systems of PDEs and recurrence relations.

For $n \geq 1$, we let \\[-0.7cm]
\begin{center}
	\begin{tabular}{ccc}
		$G^{(1)}_{2n}(p,q)=\sum_{\pi \in UD_{2n}}p^{\lrmax(\pi)}q^{\rlmax(\pi)}$ & \ \  & $G^{(2)}_{2n-1}(p,q)=\sum_{\pi \in UD_{2n-1}}p^{\lrmax(\pi)}q^{\rlmax(\pi)}$ \\[3mm]
		$G^{(3)}_{2n}(p,q)=\sum_{\pi \in DU_{2n}}p^{\lrmax(\pi)}q^{\rlmax(\pi)}$ &      & $G^{(4)}_{2n-1}(p,q)=\sum_{\pi \in DU_{2n-1}}p^{\lrmax(\pi)}q^{\rlmax(\pi)}$
	\end{tabular}
\end{center}
Along similar lines to the proof of Proposition~\ref{prop1}, we obtain the following result.

\begin{prop}\label{prop3}
	For all $n\geq 1$,
	\begin{footnotesize}
		\begin{itemize}
			\item[{\upshape 1.}] $	G^{(1)}_{2n}(p,q)
				      =\sum\limits_{\pi \in DU_{2n}}p^{\rlmax(\pi)}q^{\lrmax(\pi)}
				      =\sum\limits_{\pi \in DU_{2n}}p^{\lrmin(\pi)}q^{\rlmin(\pi)}
				      =\sum\limits_{\pi \in UD_{2n}}p^{\rlmin(\pi)}q^{\lrmin(\pi)}$.
			\item[{\upshape 2.}] $	G^{(2)}_{2n-1}(p,q)
				      =\sum\limits_{\pi \in UD_{2n-1}}p^{\rlmax(\pi)}q^{\lrmax(\pi)}
				      =\sum\limits_{\pi \in DU_{2n-1}}p^{\lrmin(\pi)}q^{\rlmin(\pi)}
				      =\sum\limits_{\pi \in DU_{2n-1}}p^{\rlmin(\pi)}q^{\lrmin(\pi)}$.
			\item[{\upshape 3.}] $	G^{(3)}_{2n}(p,q)
				      =\sum\limits_{\pi \in UD_{2n}}p^{\rlmax(\pi)}q^{\lrmax(\pi)}
				      =\sum\limits_{\pi \in UD_{2n}}p^{\lrmin(\pi)}q^{\rlmin(\pi)}
				      =\sum\limits_{\pi \in DU_{2n}}p^{\rlmin(\pi)}q^{\lrmin(\pi)}$.
			\item[{\upshape 4.}] $	G^{(4)}_{2n-1}(p,q)
				      =\sum\limits_{\pi \in DU_{2n-1}}p^{\rlmax(\pi)}q^{\lrmax(\pi)}
				      =\sum\limits_{\pi \in UD_{2n-1}}p^{\lrmin(\pi)}q^{\rlmin(\pi)}
				      =\sum\limits_{\pi \in UD_{2n-1}}p^{\rlmin(\pi)}q^{\lrmin(\pi)}$.
		\end{itemize}
	\end{footnotesize}
\end{prop}

\vspace{-0.4cm}
Note that $G^{(2)}_{2n-1}(p,q)=G^{(2)}_{2n-1}(q,p)$ and $G^{(4)}_{2n-1}(p,q)=G^{(4)}_{2n-1}(q,p)$.

By Proposition~\ref{prop3}, the study of the joint distributions of  $(\lrmax,\rlmax)$ and  $(\lrmin,\rlmin)$ in the sets $UD_{2n}$, $UD_{2n-1}$, $DU_{2n}$ and $DU_{2n-1}$  can be reduced to the study of the following four generating functions: \\[-5mm]
\begin{center}
	\begin{tabular}{ccc}
		$G^{(1)}(t,p,q)=1+\sum_{n\geq 1}G^{(1)}_{2n}(p,q)\frac{t^{2n}}{(2n)!}$ & \ \  &
		$G^{(2)}(t,p,q)=\sum_{n\geq 1}G^{(2)}_{2n-1}(p,q)\frac{t^{2n-1}}{(2n-1)!}$      \\[3mm]
		$G^{(3)}(t,p,q)=1+\sum_{n\geq 1}G^{(3)}_{2n}(p,q)\frac{t^{2n}}{(2n)!}$ &      &
		$G^{(4)}(t,p,q)=\sum_{n\geq 1}G^{(4)}_{2n-1}(p,q)\frac{t^{2n-1}}{(2n-1)!}$
	\end{tabular}
\end{center}

\begin{thm}\label{joint-dist-pattern-thm}
	We have \\[-5mm]
	\begin{eqnarray*}
		G^{(1)}(t,p,q)&=&{\color{red}1+}\int_{0}^{pqt}\left[(\sec (s/pq))^{p+q}\int_{0}^{s/q}(\sec (z/p))^{-p}dz\right]ds,\\
		G^{(2)}(t,p,q)&=&pqt+\int_{0}^{pqt}\left[(\sec (s/pq))^{p+q}\int_{0}^{s/q}(\sec (z/p))^{-p}dz   \int_{0}^{s/p}(\sec (z/q))^{-q}dz\right]ds,\\
		G^{(3)}(t,p,q)&=&{\color{red}1+}\int_{0}^{pqt}\left[(\sec (s/pq))^{p+q}\int_{0}^{s/p}(\sec (z/q))^{-q}dz\right]ds, \\
		G^{(4)}(t,p,q)&=&\int_{0}^{pqt}(\sec (z/pq))^{p+q}dz.
	\end{eqnarray*}
\end{thm}

\begin{proof} We derive the formula for $G^{(1)}(t, p, q)$ using the function $A(t,p,q)$ in Theorem~\ref{joint-dist-mmp-thm}; similar derivations for  $G^{(2)}(t, p, q)$, $G^{(3)}(t, p, q)$, $G^{(4)}(t, p, q)$ using $B(t,p,q)$, $C(t,p,q)$, $D(t,p,q)$ in Theorem~\ref{joint-dist-mmp-thm}, respectively, are omitted.

	For any $\pi \in S_n$, we have $\mmp^{(0,1,0,0)}(\pi)=n-\lrmax(\pi)$ and $\mmp^{(1,0,0,0)}(\pi)=n-\rlmax(\pi)$. Therefore,
	\begin{footnotesize}
		\begin{align*}
			A(t,p,q) & = 1+\sum_{n\geq 1}\sum_{\pi\in UD_{2n}}p^{\mmp^{(0,1,0,0)}(\pi)}q^{\mmp^{(1,0,0,0)}(\pi)}\frac{t^{2n}}{(2n)!}
			=1+\sum_{n\geq1}^{}\sum_{\pi\in UD_{2n}}p^{2n-\lrmax(\pi)}q^{2n-\rlmax(\pi)}\frac{t^{2n}}{(2n)!}                         \\
			         & =1+\sum_{n\geq 1}^{}\sum_{\pi\in UD_{2n}}(p^{-1})^{\lrmax(\pi)}(q^{-1})^{\rlmax(\pi)}\frac{(pqt)^{2n}}{(2n)!}
			=G^{(1)}(pqt, p^{-1},q^{-1}).
		\end{align*}
	\end{footnotesize}
	Hence, $G^{(1)}(t,p,q)=A(pqt,p^{-1},q^{-1})$.
\end{proof}

\section{POP-avoiding alternating permutations}\label{POPs-sec}
\noindent
In this section, we find recurrence relations for the number of permutations in $UD_{2n}$, $UD_{2n+1}$, $DU_{2n}$, and $DU_{2n+1}$ that avoid any fixed POP in Figure~\ref{pic-Bk}. Using the reverse and/or complement operations, we see that it is sufficient to consider the pattern

\vspace{-0.8cm}

\begin{center}
	\begin{tikzpicture}[scale=0.5]

		\draw [line width=1](0,0)--(1.5,1.5);
		\draw [line width=1](1,0)--(1.5,1.5);
		\draw [line width=1](3,0)--(1.5,1.5);

		\draw (0,0) node [scale=0.4, circle, draw,fill=black]{};
		\draw (1,0) node [scale=0.4, circle, draw,fill=black]{};
		\draw (3,0) node [scale=0.4, circle, draw,fill=black]{};
		\draw (1.5,1.5) node [scale=0.4, circle, draw,fill=black]{};

		\draw (1.75,0) node [scale=0.15, circle, draw,fill=black]{};
		\draw (2,0) node [scale=0.15, circle, draw,fill=black]{};
		\draw (2.25,0) node [scale=0.15, circle, draw,fill=black]{};

		\node [below] at (-0.9,1.5){\footnotesize{$\Lambda_k=$}};

		\node [below] at (0,-0.1){\footnotesize{$1$}};
		\node [below] at (1,-0.1){\footnotesize{$2$}};
		\node [below] at (3,-0.1){\footnotesize{$k-1$}};
		\node [above] at (1.5,1.5){\footnotesize{$k$}};
		\node [above] at (3.5,0.5){\footnotesize{.}};
	\end{tikzpicture}
\end{center}
\vspace{-0.7cm}
Table~\ref{tab-symmetries} gives a way of applying Theorem~\ref{POP-thm} to any POP in Figure~\ref{pic-Bk}. To avoid trivial cases, in the next theorem we assume $k\geq 3$.

\begin{thm}\label{POP-thm}
	Let  $a(2n)$, $a(2n+1)$, $b(2n)$ and $b(2n+1)$ be, respectively, the number of $\Lambda_k$-avoiding permutations in $UD_{2n}$, $UD_{2n+1}$, $DU_{2n}$ and $DU_{2n+1}$, where $k\geq 3$. Then, \\[-5mm]
	\begin{eqnarray}
		a(2n)&=&\sum_{i=1}^{\lfloor\frac{k-1}{2}\rfloor}(-1)^{i+1}{k-1\choose 2i}a(2n-2i),\label{a(2n)} \\
		a(2n+1)&=&\sum_{i=0}^{\lfloor\frac{k-2}{2}\rfloor}(-1)^{i}{k-1\choose 2i+1}a(2n-2i),\label{a(2n+1)} \\
		b(2n)&=&\sum_{i=0}^{\lfloor\frac{k-2}{2}\rfloor}(-1)^{i}{k-1\choose 2i+1}b(2n-2i-1),\label{b(2n)}  \\
		b(2n+1) &=& \sum_{i=1}^{\lfloor\frac{k-1}{2}\rfloor}(-1)^{i+1}{k-1\choose 2i}b(2n-2i+ 1),\label{b(2n+1)}
	\end{eqnarray}
	The base cases are $a(n)=b(n)=E_n$ for $n<k$, where $E_n$ is given by (\ref{E_n}).
\end{thm}

\begin{table}
	\begin{center}
		\begin{tabular}{c||c|c|c|c}
			            &
			\begin{tikzpicture}[scale=0.4]

				\draw [line width=1](0,0)--(1.5,1.5);
				\draw [line width=1](1,0)--(1.5,1.5);
				\draw [line width=1](3,0)--(1.5,1.5);

				\draw (0,0) node [scale=0.4, circle, draw,fill=black]{};
				\draw (1,0) node [scale=0.4, circle, draw,fill=black]{};
				\draw (3,0) node [scale=0.4, circle, draw,fill=black]{};
				\draw (1.5,1.5) node [scale=0.4, circle, draw,fill=black]{};

				\draw (1.75,0) node [scale=0.15, circle, draw,fill=black]{};
				\draw (2,0) node [scale=0.15, circle, draw,fill=black]{};
				\draw (2.25,0) node [scale=0.15, circle, draw,fill=black]{};

				\node [below] at (0,-0.1){\tiny{$1$}};
				\node [below] at (1,-0.1){\tiny{$2$}};
				\node [below] at (3,-0.1){\tiny{$k-1$}};
				\node [above] at (1.5,1.5){\tiny{$k$}};
			\end{tikzpicture}

			            &
			\begin{tikzpicture}[scale=0.4]

				\draw [line width=1](0,0)--(1.5,1.5);
				\draw [line width=1](1,0)--(1.5,1.5);
				\draw [line width=1](3,0)--(1.5,1.5);

				\draw (0,0) node [scale=0.4, circle, draw,fill=black]{};
				\draw (1,0) node [scale=0.4, circle, draw,fill=black]{};
				\draw (3,0) node [scale=0.4, circle, draw,fill=black]{};
				\draw (1.5,1.5) node [scale=0.4, circle, draw,fill=black]{};

				\draw (1.75,0) node [scale=0.15, circle, draw,fill=black]{};
				\draw (2,0) node [scale=0.15, circle, draw,fill=black]{};
				\draw (2.25,0) node [scale=0.15, circle, draw,fill=black]{};

				\node [below] at (0,-0.1){\tiny{$2$}};
				\node [below] at (1,-0.1){\tiny{$3$}};
				\node [below] at (3,-0.1){\tiny{$k$}};
				\node [above] at (1.5,1.5){\tiny{$1$}};
			\end{tikzpicture}

			            &

			\begin{tikzpicture}[scale=0.4]

				\draw [line width=1](0,1.5)--(1.5,0);
				\draw [line width=1](1,1.5)--(1.5,0);
				\draw [line width=1](3,1.5)--(1.5,0);

				\draw (1.5,0) node [scale=0.4, circle, draw,fill=black]{};
				\draw (1,1.5) node [scale=0.4, circle, draw,fill=black]{};
				\draw (3,1.5) node [scale=0.4, circle, draw,fill=black]{};
				\draw (0,1.5) node [scale=0.4, circle, draw,fill=black]{};

				\draw (1.75,1.5) node [scale=0.15, circle, draw,fill=black]{};
				\draw (2,1.5) node [scale=0.15, circle, draw,fill=black]{};
				\draw (2.25,1.5) node [scale=0.15, circle, draw,fill=black]{};

				\node [above] at (0,1.5){\tiny{$2$}};
				\node [above] at (1,1.5){\tiny{$3$}};
				\node [above] at (3,1.5){\tiny{$k$}};
				\node [below] at (1.5,0){\tiny{$1$}};
			\end{tikzpicture}

			            &
			\begin{tikzpicture}[scale=0.4]

				\draw [line width=1](0,1.5)--(1.5,0);
				\draw [line width=1](1,1.5)--(1.5,0);
				\draw [line width=1](3,1.5)--(1.5,0);

				\draw (1.5,0) node [scale=0.4, circle, draw,fill=black]{};
				\draw (1,1.5) node [scale=0.4, circle, draw,fill=black]{};
				\draw (3,1.5) node [scale=0.4, circle, draw,fill=black]{};
				\draw (0,1.5) node [scale=0.4, circle, draw,fill=black]{};

				\draw (1.75,1.5) node [scale=0.15, circle, draw,fill=black]{};
				\draw (2,1.5) node [scale=0.15, circle, draw,fill=black]{};
				\draw (2.25,1.5) node [scale=0.15, circle, draw,fill=black]{};

				\node [above] at (0,1.5){\tiny{$1$}};
				\node [above] at (1,1.5){\tiny{$2$}};
				\node [above] at (3,1.5){\tiny{$k-1$}};
				\node [below] at (1.5,0){\tiny{$k$}};
			\end{tikzpicture}
			\\[-2mm]
			\hline
			\hline
			$UD_{2n}$   & $a(2n)$   & $b(2n)$   & $b(2n)$   & $a(2n)$   \\
			\hline
			$UD_{2n+1}$ & $a(2n+1)$ & $a(2n+1)$ & $b(2n+1)$ & $b(2n+1)$ \\
			\hline
			$DU_{2n}$   & $b(2n)$   & $a(2n)$   & $a(2n)$   & $b(2n)$   \\
			\hline
			$DU_{2n+1}$ & $b(2n+1)$ & $b(2n+1)$ & $a(2n+1)$ & $a(2n+1)$
		\end{tabular}
	\end{center}
	\caption{Application of the results in Theorem~\ref{POP-thm} to POPs in Figure~\ref{pic-Bk}. Columns 3, 4 and 5 are obtained, respectively, by applying to alternating permutations reverse, complement and composition of reverse and complement operations.}\label{tab-symmetries}
\end{table}

\begin{proof}
	To derive (\ref{a(2n)}), let $\pi=\pi_1\ldots\pi_{2n}\in UD_{2n}$ (so that $\pi_{2n-1}<\pi_{2n}$) and $\pi$ avoids $\Lambda_k$. We will use the inclusion-exclusion method to derive the recurrence. Note that $\pi_{2n-1}, \pi_{2n}\in\{1,\ldots,k-1\}$ or else an occurrence of  $\Lambda_k$ involving $\pi_{2n}$ can be found in $\pi$, which is impossible. Hence, to form $\pi$, we can choose $\pi_{2n-1}\pi_{2n}$ in ${k-1\choose 2}$ ways and let $\pi_1\ldots\pi_{2n-2}$ be any permutation in $UD_{2n-2}$ (there are $a(2n-2)$ choices to form such a permutation). However, there are non-valid choices to construct $\pi$ in this way, namely, in some cases $\pi_{2n-3}\pi_{2n-2}\pi_{2n-1}\pi_{2n}$, formed from elements in $\{1,\ldots,k-1\}$, will be in increasing order, so we need to subtract those permutations. The number of such permutations is given by ${k-1\choose 4}$ choices for $\pi_{2n-3}\pi_{2n-2}\pi_{2n-1}\pi_{2n}$ multiplied by $a(2n-4)$ choices for $\pi_1\ldots\pi_{2n-4}$ in $UD_{2n-4}$. However, we subtracted too many permutations, and the permutations ending with six increasing elements from $\{1,\ldots,k-1\}$ need to be added back. Continuing in this way, we obtain (\ref{a(2n)}).

	To derive (\ref{a(2n+1)}), let $\pi=\pi_1\ldots\pi_{2n+1}\in UD_{2n+1}$ (so that $\pi_{2n}>\pi_{2n+1}$) and $\pi$ avoids $\Lambda_k$. Once again, we will use the inclusion-exclusion method. Note that $\pi_{2n+1}\in\{1,\ldots,k-1\}$ or else an occurrence of  $\Lambda_k$ involving $\pi_{2n+1}$ can be found in $\pi$, which is impossible. Hence, to form $\pi$, we can choose $\pi_{2n+1}$ in ${k-1\choose 1}$ ways and let $\pi_1\ldots\pi_{2n}$ be any permutation in $UD_{2n}$ (there are $a(2n)$ choices to form such a permutation). However, there are non-valid choices to construct $\pi$ in this way, namely, in some cases $\pi_{2n-1}\pi_{2n}\pi_{2n+1}$, formed from elements in $\{1,\ldots,k-1\}$, will be in increasing order, so we need to subtract those permutations. The number of such permutations is given by ${k-1\choose 3}$ choices for $\pi_{2n-1}\pi_{2n}\pi_{2n+1}$ multiplied by $a(2n-2)$ choices for $\pi_1\ldots\pi_{2n-2}$ in $UD_{2n-2}$. The rest of our arguments are similar to the derivation of (\ref{a(2n)}).

	Our proofs of (\ref{b(2n)}) and (\ref{b(2n+1)}) are similar to those of  (\ref{a(2n)}) and  (\ref{a(2n+1)}) and hence are omitted.
\end{proof}

\section{Concluding remarks}\label{Concluding-sec}
\noindent
In this paper, we derive the (joint) distributions of maxima and minima statistics for up-down and down-up permutations of even and odd lengths. This refines classic enumeration results of Andr\'e~\cite{Andre1,Andre2} and introduces new $q$-analogues and $(p,q)$-analogues for the number of alternating permutations.  Also, we confirm Callan's conjecture that the number of up-down permutations of even length, fixed by reverse and complement, equals the Springer numbers. Our approach allows us to propose two $q$-analogues and a $(p,q)$-analogue for the Springer numbers. Moreover, we enumerate alternating permutations that avoid any POP presented in Figure~\ref{pic-Bk}.

For future research directions, one could generalize our results on the joint distribution of alternating permutations in Theorem~\ref{joint-dist-pattern-thm} by exploring the simultaneous control of three or four maxima/minima statistics, akin to the approach in \cite{CKZ} for separable permutations. Additionally, initiating studies on maxima/minima statistics for POP-avoiding alternating permutations, in particular, those enumerated in Theorem~\ref{POP-thm}, would be interesting. Lastly, can our combinatorial interpretation of the Springer numbers in terms of up-down permutations of even length, invariant under the composition of reverse and complement operations, be employed to find the meaning of the statistics recorded by $q$ in the following $q$-analogues of the generating function (\ref{Springer numbers})? The first of these $q$-analogues appears in \cite{EuFu}, while the last one resembles $(\sec (t))^{q}$ given in Theorem~\ref{dist-single-pattern-thm}, which provides the distribution of right-to-left maxima on up-down permutations of even length.
\begin{eqnarray*}
	\frac{1}{\cos t - q\sin t} &= &1 + qt + (1+2q^2)\frac{t^2}{2!}+(5q+6q^3)\frac{t^3}{3!}+(5+28q^2+24q^4)\frac{t^4}{4!}+\cdots\\
	\frac{1}{\cos t - \sin qt} &= & 1 + qt +(1+2q^2)\frac{t^2}{2!}+(6q+5q^3)\frac{t^3}{3!}+(5+36q^2+16q^4)\frac{t^4}{4!}+\cdots \\
	\frac{1}{\cos qt - \sin t} &= & 1 + t +(2+q^2)\frac{t^2}{2!}+(5+6q^2)\frac{t^3}{3!}+(16+36q^2+5q^4)\frac{t^4}{4!}+\cdots \\
	\frac{1}{(\cos t - \sin t)^{q}}&= & 1 +q t +(2q+q^2)\frac{t^2}{2!}+(4q+6q^2+q^3)\frac{t^3}{3!}+(16q+28q^2+12q^3+q^4)\frac{t^4}{4!}+\cdots \\
\end{eqnarray*}

\noindent
{\bf Acknowledgments.}  We are grateful to the anonymous referees for their helpful comments and suggestions. We are also grateful to Jesse Geneson for pointing out a typo in the formula for $F^{(3)}(t,q)$ in Theorem~3.2 of the published version of this paper. The work of Philip B. Zhang was supported by the National Natural Science Foundation of China (No.\ 12171362).

\end{document}